\documentclass{article}

\usepackage{graphicx}
\usepackage{amscd}
\usepackage{amsmath}
\usepackage{amsthm}
\usepackage{amsfonts}
\usepackage{amssymb}

\newtheorem{theorem}{Theorem}

\newtheorem*{acknowledgement}{Acknowledgement}

\newtheorem{corollary}{Corollary}

\newtheorem{example}{Example}

\newtheorem{lemma}{Lemma}

\newtheorem{proposition}{Proposition}

\title{Beauville surfaces with abelian Beauville group}

\author{Gabino Gonz\'{a}lez-Diez \and Gareth A. Jones \and David Torres-Teigell}

\begin{document}

\maketitle

\begin{abstract}A Beauville surface is a rigid surface of general type arising as a quotient of a product of curves $C_{1}$, $C_{2}$ of genera $g_{1},g_{2}\ge 2$ by the free action of a finite group $G$. In this paper we study those Beauville surfaces for which $G$ is abelian (so that $G\cong \mathbb{Z}_{n}^{2}$ with $\gcd(n,6)=1$ by a result of Catanese). For each such $n$ we are able to describe all such surfaces, give a formula for the number of their isomorphism classes and identify their possible automorphism groups. This explicit description also allows us to observe that such surfaces are all defined over $\mathbb{Q}$.


\

\end{abstract}

\section{Introduction and statement of results}\label{sec:Introduction}

A complex surface $S$ is said to be \emph{isogenous to a product} if it is isomorphic to the quotient of a product of curves $C_{1}\times C_{2}$ of genus $g_{1}=g(C_{1}),g_{2}=g(C_{2})\ge 1$ by the free action of a finite subgroup $G$ of $\mathrm{Aut}(C_{1}\times C_{2})$, the automorphism group of $C_{1}\times C_{2}$.

If additionally the curves satisfy $g_{1},g_{2}\ge 2$ we say that $S$ is \emph{isogenous to a higher product}. We will always assume that $g_{1}\le g_{2}$.

It is not difficult to see that an element of $\mathrm{Aut}(C_{1}\times C_{2})$ either preserves each curve or interchanges them, the latter being possible only if $C_{1}\cong C_{2}$. Consequently one says that $S$ is of \emph{unmixed type} if no element of $G$ interchanges factors, and of \emph{mixed type} otherwise.

On the other hand if two elements $g,h\in G$ interchange factors it is clear that $gh$ belongs to the subgroup $G^{0}<G$ of factor-preserving elements, and therefore $[G:G^{0}]=2$.

There is always a minimal realization of $S$ in the sense that $G^{0}$ acts faithfully on each factor $C_{i}$. This is because if for instance $G^{0}$ did not act faithfully on $C_{1}$, so that there exists a subgroup $G'\lhd G$ acting trivially on $C_{1}$, then we could write
\[ S \cong \frac{C_{1}\times C_{2}}{G}=\frac{C_{1}/G' \times C_{2}/G'}{G/G'}=\frac{C_{1}\times (C_{2}/G')}{G/G'}\]Hence, from now on we will always assume that the realization $S\cong (C_{1}\times C_{2})/G$ is minimal.

\
%
%

A Beauville surfaces is a particular type of surface isogenous to a product. These surfaces were introduced by Catanese in~\cite{Cat} following a construction by Beauville in~\cite{Bea} (see Example~\ref{ej:Bea} below), and they have been subsequently studied by himself, Bauer and Grunewald (\cite{BaCa}, \cite{BaCaGr05}, \cite{BaCaGr07}), somewhat later by Fuertes, Gonz\'{a}lez-Diez and Jaikin (\cite{FuGo}, \cite{FuGoJa}), and more recently by Jones, Penegini, Garion, Larsen, Lubotzky, Guralnick, Malle, Fairbairn, Magaard and Parker (\cite{FuJo}, \cite{Jon}, \cite{GaLaLu}, \cite{GaPe}, \cite{GuMa}, \cite{FMP}).

A \emph{Beauville surface} is a compact complex surface $S$ isogenous to a higher product, $S \cong (C_{1}\times C_{2}) /G$, where $G^{0}$ acts on each of the curves in such a way that $C_{1}/G^{0}\cong\mathbb{P}^{1}$, $C_{2}/G^{0}\cong\mathbb{P}^{1}$, and the natural projections $C_{1}\longrightarrow\mathbb{P}^{1}$ and $C_{2}\longrightarrow\mathbb{P}^{1}$ ramify over three values.

Catanese proved that both the group $G$ and the curves $C_{1}$, $C_{2}$ --- hence also the genera $g_{1}$, $g_{2}$ --- are invariants of the Beauville surface. Consequently we will say that the surface $S$ has \emph{(Beauville) group} $G$ and \emph{covering product} $C_{1}\times C_{2}$.

Here we will consider only unmixed Beauville surfaces. In this case there exist isomorphisms $\varphi_{i}:G\longrightarrow H_{i}\le\mathrm{Aut}(C_{i})$ so that the action of an element $g\in G$ on a point $(p_{1},p_{2})\in C_{1}\times C_{2}$ is given by $g(p_{1},p_{2})=(\varphi_{1}(g)(p_{1}),\varphi_{2}(g)(p_{2}))$. From this point of view giving an unmixed Beauville surface amounts to specifying an isomorphism $\phi=\varphi_{2}\circ\varphi_{1}^{-1}$ between groups $H_{1}\le\mathrm{Aut}(C_{1})$ and $H_{2}\le\mathrm{Aut}(C_{2})$ satisfying the following two conditions
    \begin{enumerate}
        \item[(C1)] $C_{i}/H_{i}$ is an orbifold of genus zero with three branching values, and
        \item[(C2)] for any non-identity element $h_{1}\in H_{1}$ which fixes a point on $C_{1}$ the element $h_{2}=\phi(h_{1})\in H_{2}$ acts freely on $C_{2}$ (so that the action $h(p_{1},p_{2})=(h_{1}(p_{1}),h_{2}(p_{2}))$ is free on $C_{1}\times C_{2}$).
    \end{enumerate}

It turns out that the problem of deciding whether a finite group is an unmixed Beauville group is purely group theoretical \cite{Cat}. This occurs if and only if $G$ admits two triples of generators $(a_{1},b_{1},c_{1})$ and $(a_{2},b_{2},c_{2})$ such that
\begin{itemize}
    \item[(i)] $a_{1}b_{1}c_{1}=a_{2}b_{2}c_{2}=1$;
    \item[(ii)] $\dfrac{1}{\mathrm{ord}(a_{i})}+\dfrac{1}{\mathrm{ord}(b_{i})}+\dfrac{1}{\mathrm{ord}(c_{i})}<1$, $i=1,2$;
    \item[(iii)] $\Sigma(a_{1},b_{1},c_{1})\cap\Sigma(a_{2},b_{2},c_{2})=\{\mathrm{Id}_{G}\}$, where we define
        \[\Sigma(a,b,c)=\bigg\{\bigcup_{g\in G} \Big(g\langle{a}\rangle g^{-1} \cup g\langle{b}\rangle g^{-1} \cup g\langle{c}\rangle g^{-1} \Big)\bigg\}\]
\end{itemize}

This is a very useful characterization since it allows the use of computer programs such as GAP or MAGMA to solve the problem for groups of low order.

The first example of a Beauville surface was given by A.~Beauville some twenty years before the term was coined by Catanese. It appears as an exercise at the end of his book on complex surfaces \cite{Bea} as an example of a complex surface of general type whose two geometrical invariants $q$ and $p_{g}$ vanish.

\begin{example}[Beauville]\label{ej:Bea} Let us denote by $\mathbb{Z}_{n}=\mathbb{Z}/n\mathbb{Z}$ the group of integers modulo $n$ and by $F_{n}$ the Fermat curve of degree $n$,
    \[F_{n}=\{[x:y:z]\in\mathbb{P}^{2}(\mathbb{C})\ :\ x^{n}+y^{n}+z^{n}=0\}.\]For $n=5$ the group $G=\mathbb{Z}_{5}^{2}$ acts freely on $F_{5}\times F_{5}$ in the following way: for each $(\alpha,\beta)\in G$ we put
    \[(\alpha,\beta)\left(\left[\begin{tabular}{c}$x_{1}$\\$y_{1}$\\$z_{1}$\\
                 \end{tabular}\right]
     ,\left[\begin{tabular}{c}$x_{2}$\\$y_{2}$\\$z_{2}$\\
            \end{tabular}\right]\right)= \left(\left[\begin{tabular}{c}$\xi^{\alpha}x_{1}$\\$\xi^{\beta}y_{1}$\\$z_{1}$\\
                    \end{tabular}\right]
     ,\left[\begin{tabular}{c} $\xi^{\alpha+3\beta}x_{2}$\\$\xi^{2\alpha+4\beta}y_{2}$\\$z_{2}$\\
            \end{tabular}\right]\right) \]where $\xi=e^{2\pi i/5}$.
It is easy to see that the surface $S:=(F_{5}\times F_{5})/\mathbb{Z}_{5}^{2}$ is an unmixed Beauville surface.
\end{example}

Obviously the formula
    \[(\alpha,\beta)\left([x:y:z]\right)=[\xi^{\alpha}x:\xi^{\beta}y:z]\quad(\xi=e^{2\pi i/n})\]
    also defines an action of the group $\mathbb{Z}_{n}^{2}$ on $F_{n}$ when $n>5$. The quotient $F_{n}/\mathbb{Z}_{n}^{2}$ is always an orbifold of genus zero with three branching values, all of them of order $n$. Therefore any group automorphism $\phi:\mathbb{Z}_{n}^{2}\longrightarrow  \mathbb{Z}_{n}^{2}$ satisfying condition (C2) above gives rise to a Beauville surface, something that can occur only when $\gcd(n,6)=1$ (see~\cite{Cat}). Since any such $\phi$ is necessarily of the form
    \[\phi=\phi_{A}:
    \left(
        \begin{array}{c}
            \alpha \\
            \beta \\
        \end{array}
    \right)\longmapsto
    A\cdot\left(
        \begin{array}{c}
            \alpha \\
            \beta \\
        \end{array}
    \right)\quad\mbox{for some }\quad A=\left(\begin{array}{cc} a & b \\ c & d \end{array}\right)\in\mathrm{GL}_{2}(\mathbb{Z}_{n}),\]any surface obtained in this way is of the form $S_{A}^{n}=(F_{n}\times F_{n})/G_{A}$, where $G_{A}$ is the subgroup of $\mathrm{Aut}(F_{n}\times F_{n})$ defined by
    \[G_{A}=\left\{ \left( \left(\begin{array}{c} \alpha \\ \beta \\ \end{array}\right),A \left(\begin{array}{c} \alpha \\ \beta \\ \end{array}\right) \right)\ :\ \left(\begin{array}{c} \alpha \\ \beta \\ \end{array}\right)\in G \right\}\cong G\]
In this way the action of an element $(\alpha,\beta)\in\mathbb{Z}_{n}^{2}$ on $F_{n}\times F_{n}$ is given by
    \begin{equation*}(\alpha,\beta)
    \left(\left[\begin{tabular}{c}$x_{1}$\\$y_{1}$\\$z_{1}$\\
                 \end{tabular}\right]
     ,\left[\begin{tabular}{c}$x_{2}$\\$y_{2}$\\$z_{2}$\\
            \end{tabular}\right]\right) =
    \left(\left[\begin{tabular}{c}$\xi^{\alpha}x_{1}$\\$\xi^{\beta}y_{1}$\\$z_{1}$\\
                    \end{tabular}\right]
     ,\left[\begin{tabular}{c} $\xi^{a\alpha+b\beta}x_{2}$\\$\xi^{c\alpha+d\beta}y_{2}$\\$z_{2}$\\
            \end{tabular}\right]\right).\end{equation*}Thus the surface constructed in Example~\ref{ej:Bea} is $S_{A}^{5}$ for $A=\bigl( \begin{smallmatrix}1&3\\ 2&4\end{smallmatrix} \bigr)$.

It is known that, even though in the definition of a Beauville surface both curves $C_{1}$ and $C_{2}$ are only required to have genus greater than 1, in fact both genera have to be greater than 5 (see~\cite{FuGoJa}). Thus Beauville's example $S_{A}^{5}$ above reaches the minimum possible bigenus $(g_{1},g_{2})=(6,6)$.

In this paper we prove the following facts:
\begin{enumerate}
    \item Each Beauville surface with an abelian group $G$ is isomorphic to one of the form $S^{n}_{A}$ (Theorem~\ref{thm:abelian}), and is defined over $\mathbb Q$ (Corollary~\ref{cor:moduli}).
    \item The number $\Theta(n)$ of isomorphism classes of Beauville surfaces which have Beauville group $\mathbb{Z}_{n}^{2}$ (where $\gcd(n,6)=1$) is given by the formula in Theorem~\ref{thm:numbersurfaces}. One consequence is that asymptotically $\Theta(n)\sim n^{4}/72$ for prime powers $n$. Another is that $\Theta(5)=1$, which means that Beauville's original surface $S_{A}^{5}$ mentioned above is the only Beauville surface with group $\mathbb{Z}_{5}^{2}$.
    \item The automorphism group of $S_{A}^{n}$ is either $\mathbb{Z}_{n}^{2}$, or an extension of $\mathbb{Z}_{n}^{2}$ by one of the cyclic groups $\mathbb{Z}_{2}$, $\mathbb{Z}_{3}$ and $\mathbb{Z}_{6}$ or by the symmetric group $\mathcal{S}_{3}$ (Proposition~\ref{prop:automorphismgroup}).
\end{enumerate}

In~\cite{BaCaGr05} Bauer, Catanese and Grunewald have given a lower bound for the asymptotic behaviour of $\Theta(n)$ as $n\to\infty$. More recently, Garion and Penegini in~\cite{GaPe} have given upper and lower bounds for $\Theta(n)$ for each $n$, and in Theorem~2 we extend their results by giving an explicit formula for $\Theta(n)$.

\section{Unmixed Beauville structures on the product of Fermat curves}

It was proved in~\cite{Cat} that if $S=(C_{1}\times C_{2})/G$ is a Beauville surface with abelian group $G$ then $G=\mathbb{Z}^{2}_{n}$ with $\gcd(n,6)=1$, and in~\cite{FuGoJa} that in that case $C_{1} = C_{2} =F_{n}$, the Fermat curve of degree $n$. Combining these two facts one gets the following slightly more precise result:

\begin{theorem}\label{thm:abelian}Any Beauville surface with abelian Beauville group is isomorphic to one of the form $S^{n}_{A}=(F_{n}\times F_{n})/G_{A}$.
\end{theorem}

\begin{proof} In view of the results in~\cite{Cat} and~\cite{FuGoJa} mentioned above it is sufficient to observe that $\mathrm{Aut}(F_{n})$ possesses a unique subgroup isomorphic to $Z_{n}^{2}$ satisfying condition (C1) in our definition of Beauville surface, and therefore any Beauville surface with group $\mathbb{Z}_{n}^{2}$ is determined by an isomorphism $A: \mathbb{Z}_{n}^{2}\longrightarrow\mathbb{Z}_{n}^{2}$ satisfying condition (C2).
\end{proof}

This explicit  description of Beauville surfaces with abelian group yields the following:

\begin{corollary}\label{cor:moduli}
Each Beauville surface with abelian group is defined over $\mathbb Q$.
\end{corollary}

\begin{proof}
Since the curves $F_n$ are obviously defined over $\mathbb Q$ we only have to show  that so is the group $G_A$. That is, we have to show that the obvious action of the absolute Galois group ${\rm Gal}(\overline{\mathbb Q}/{\mathbb Q})$ on automorphisms of $F_n\times F_n$ leaves the group $G_A$ setwise invariant.

Now, let $\sigma$ be an arbitrary element of ${\rm Gal}(\overline{\mathbb Q}/{\mathbb Q})$, and suppose that $\sigma(\xi)= \xi^ j$;  then, clearly, $\sigma$ transforms the action of an element $(\alpha, \beta)$ of $G_A$ into the action of another element of $G_A$, namely $(j\alpha,j\beta)$.
\end{proof}

\noindent{\bf Remark~1.} In regard to the fields of definition of quotient varieties by abelian groups we draw the reader's attention to the interesting  Corollary~1.9 in the article~\cite{BaCa08} by Bauer and Catanese.

\

We would like to find which matrices $A=\bigl( \begin{smallmatrix}a&b\\ c&d\end{smallmatrix} \bigr) \in \mathrm{GL}_{2}(\mathbb{Z}_{n})$ define Beauville surfaces $S_{A}^{n}$, that is, we would like to characterize those matrices $A$ such that the group $G_{A}$ defined above acts freely on $F_{n}\times F_{n}$. Note that the Beauville surface $S_{A}^{n}$ corresponds, in terms of triples of generators, to a pair of ordered triples satisfying conditions (i) to (iii) where, without loss, we can take the first triple to be the standard triple $((1,0), (0,1), (-1,-1))$. Then the second triple is $((a,c),(b,d),(-a-b,-c-d))$, obtained from the first one by the automorphism represented by the matrix $A\in\mathrm{GL}_2(\mathbb{Z}_n)$.

With the previous notation, the elements $(\alpha,\beta)$ fixing points in the first component are precisely those of the form $(k,0)$, $(0,k)$ and $(k,k)$. Equivalently the elements $(\alpha,\beta)$ that fix points in the second curve are those such that $a\alpha+b\beta=0$, $c\alpha+d\beta=0$ or $a\alpha+b\beta=c\alpha+d\beta$. We obtain the following:

\begin{lemma}\label{lem:conditions}Let $A=\bigl( \begin{smallmatrix}a&b\\ c&d\end{smallmatrix} \bigr) \in \mathrm{GL}_{2}(\mathbb{Z}_{n})$. The group $G_{A}$ defined above acts freely on the product $F_{n}\times F_{n}$ if and only if the following conditions hold:
    \begin{equation}\label{eq:freeactions}a\,,\,b\,,\,c\,,\,d\,,\,a+b\,,\,c+d\,,\,a-c\,,\,b-d\,,\,a+b-c-d\in U(\mathbb{Z}_{n})\end{equation}where $U(\mathbb{Z}_{n})$ is the group of units of $\mathbb{Z}_{n}$.
\end{lemma}

\begin{proof}Let $(\alpha,\beta)\in\mathbb{Z}_{n}$ be an element fixing some point of the first curve.

If $(\alpha,\beta)=(k,0)$ the action on the second curve is given by $(ak,ck)$, which fixes no point if and only if $ak\neq 0$, $ck\neq 0$, $ak\neq ck$. This will be true for all $k$ if and only if $a,c,a-c\in U(\mathbb{Z}_{n})$.

Arguing the same way with elements of the form $(\alpha,\beta)=(0,k)$ and $(\alpha,\beta)=(k,k)$ one obtains the result.
\end{proof}

Note that from this lemma we can deduce the already mentioned fact, due to Catanese, that $\gcd(n,6)$ must be 1 for $\mathbb{Z}_{n}^{2}$ to admit a Beauville structure.

If $n$ is even, the conditions $a, b, a+b\in U(\mathbb{Z}_{n})$ cannot hold simultaneously. On the other hand if $n$ is a multiple of 3 then necessarily $a\equiv b\bmod{3}$ and $c\equiv d\bmod{3}$, but then the matrix $A$ is not invertible.

We will denote by $\mathfrak{F}_{n}$ the set of matrices in $\mathrm{GL}_{2}(\mathbb{Z}_{n})$ satisfying the conditions in \eqref{eq:freeactions}, and will write $e^{A}_{(\alpha,\beta)}$ for the element of $G_{A}$ corresponding to $(\alpha,\beta)\in\mathbb{Z}_{n}^{2}$. In $\mathbb{Z}_{5}^{2}$, for instance, there are 24 different matrices satisfying these conditions.

It follows from Theorem~\ref{thm:abelian} that $\mathrm{Aut}(F_{n}\times F_{n})$ acts on the set of groups $\{G_{A}\ :\ A\in\mathfrak{F}_{n}\}$ by conjugation. Furthermore, two such groups are in the same orbit if and only if the corresponding Beauville surfaces are isomorphic. This is because any isomorphism $\phi:S_{A}^{n}\longrightarrow S_{A'}^{n}$ lifts to an automorphism $\Phi$ of the product of curves, yielding a commutative diagram
    \[\begin{tabular}{rcccl}
        & $F_{n}\times F_{n}$ & $\xrightarrow{\quad\Phi\quad}$ & $F_{n}\times F_{n}$ & \\
        &\rotatebox{90}{$\longleftarrow$} & & \rotatebox{90}{$\longleftarrow$}& \\
        $S=$ & $\dfrac{F_{n}\times F_{n}}{G_{A}}$ & $\xrightarrow{\quad\phi\quad}$ & $\dfrac{F_{n}\times F_{n}}{G_{A'}}$& $=S'$ \\
    \end{tabular}\]such that $\Phi G_{A}\Phi^{-1}=G_{A'}$ (see \cite{Cat}).

It is well known that the automorphism group of $F_{n}$ is $\mathbb{Z}_{n}^{2}\rtimes \mathcal{S}_{3}$ and so $\mathrm{Aut}(F_{n}\times F_{n})=\langle\ \mathbb{Z}_{n}^{2}\times \mathbb{Z}_{n}^{2},\ \mathcal{S}_{3}\times\mathcal{S}_{3},\ J \ \rangle$, where $J(\mu_{1},\mu_{2})=(\mu_{2},\mu_{1})$ and the groups $\mathbb{Z}_{n}^{2}\times \mathbb{Z}_{n}^{2}$ and $\mathcal{S}_{3}\times\mathcal{S}_{3}$ act separately on each factor, the first one by multiplying the homogeneous coordinates by $n$-th roots of unity and the second one by permuting them. We now note that every element of $\mathbb{Z}_{n}^{2}\times \mathbb{Z}_{n}^{2}$ fixes each element $e_{(\alpha,\beta)}^{A}\in G_{A}$ by conjugation, so the action of this subgroup on the set $\{G_{A}\}$ is trivial. We can therefore restrict our attention to the quotient group $W = \mathrm{Aut}(F_{n}\times F_{n})/(\mathbb{Z}_{n}^{2}\times \mathbb{Z}_{n}^{2})$, which is a semidirect product $(\mathcal{S}_{3}\times\mathcal{S}_{3})\rtimes\langle J\rangle$ of $\mathcal{S}_{3}\times\mathcal{S}_{3}$ by $\langle J\rangle\cong\mathcal{S}_2$, with the complement $\mathcal{S}_2$ transposing the direct two factors $\mathcal{S}_3$ by conjugation. Thus $W$ is the wreath  product $\mathcal{S}_3\wr\mathcal{S}_2$ of $\mathcal{S}_3$ by $\mathcal{S}_2$.

The group $\mathcal{S}_{3}$ can be viewed as a subgroup of $\mathrm{GL}_{2}(\mathbb{Z}_{n})$ via the group monomorphism:
    \[\begin{tabular}{cccc}
      $M:$ & $\mathcal{S}_{3}$ & $\longrightarrow$ & $\mathrm{GL}_{2}(\mathbb{Z}_{n})$ \\
         & $\tau$ & $\longmapsto$ & $M_{\tau}$
    \end{tabular}\]determined by $M_{\sigma_{1}}=\bigl( \begin{smallmatrix}-1&1\\ -1&0\end{smallmatrix} \bigr)$ for $\sigma_{1}=(1,3,2)$ and $M_{\sigma_{2}}=\bigl( \begin{smallmatrix}0&1\\ 1&0\end{smallmatrix} \bigr)$ for $\sigma_{2}=(1,2)$.

\begin{lemma}Let $\Phi=(\tau_{1},\tau_{2})\in\mathcal{S}_{3}\times\mathcal{S}_{3}$ be a factor-preserving automorphism. Then $\Phi G_{A}\Phi^{-1}=G_{A'}$, where $A'=M_{\tau_{2}}AM_{\tau_{1}}^{-1}$.

On the other hand, if $\Phi'=\Phi\circ J$ is an automorphism interchanging factors, then $\Phi' G_{A}\Phi'^{-1}=G_{A'}$, where $A'=M_{\tau_{2}}A^{-1}M_{\tau_{1}}^{-1}$.
\end{lemma}

\begin{proof}First we note that if $A\in\mathfrak{F}_{n}$, then $A'=M_{\tau_{2}}A^{\pm 1}M_{\tau_{1}}^{-1}$ belongs to $\mathfrak{F}_{n}$ as well. We will now see how the group $\mathrm{Aut}(F_{n}\times F_{n})$ acts by conjugation on the elements of $G_{A}$.

To see how an element $(\tau_{1},\tau_{2})\in\mathcal{S}_{3}\times\mathcal{S}_{3}$ acts on $G_{A}$ we will write $(\tau_{1},\tau_{2})=(\tau_{1},\mathrm{Id})\circ(\mathrm{Id},\tau_{2})$, with $\tau_{1},\tau_{2}\in\mathcal{S}_{3}$. If $\Phi=(\mathrm{Id},\sigma_{1})$ then
    \[\Phi e^{A}_{(\alpha,\beta)}\Phi^{-1}
    \left(\left[\begin{tabular}{c}$x_{1}$\\$y_{1}$\\$z_{1}$\\
                 \end{tabular}\right]
     ,\left[\begin{tabular}{c}$x_{2}$\\$y_{2}$\\$z_{2}$\\
            \end{tabular}\right]\right) =
    \left(\left[\begin{tabular}{c}$\xi^{\alpha}x_{1}$\\$\xi^{\beta}y_{1}$\\$z_{1}$\\
                    \end{tabular}\right]
     ,\left[\begin{tabular}{c} $\xi^{(c-a)\alpha+(d-b)\beta}x_{2}$\\$\xi^{-a\alpha-b\beta}y_{2}$\\$z_{2}$\\
            \end{tabular}\right]\right)\]so $\Phi e^{A}_{(\alpha,\beta)}\Phi^{-1}= e^{A'}_{(\alpha,\beta)}$ with $A'=\left(\begin{array}{rr} c-a & d-b \\ -a & -b \end{array}\right)=M_{\sigma_{1}}A$. In the same way we have that for $\Phi=(\mathrm{Id},\sigma_{2})$, conjugation yields
    \[\Phi e^{A}_{(\alpha,\beta)}\Phi^{-1}
    \left(\left[\begin{tabular}{c}$x_{1}$\\$y_{1}$\\$z_{1}$\\
                 \end{tabular}\right]
     ,\left[\begin{tabular}{c}$x_{2}$\\$y_{2}$\\$z_{2}$\\
            \end{tabular}\right]\right) =
    \left(\left[\begin{tabular}{c}$\xi^{\alpha}x_{1}$\\$\xi^{\beta}y_{1}$\\$z_{1}$\\
                    \end{tabular}\right]
     ,\left[\begin{tabular}{c} $\xi^{c\alpha+d\beta}x_{2}$\\$\xi^{a\alpha+b\beta}y_{2}$\\$z_{2}$\\
     \end{tabular}\right]\right).\]
Thus $\Phi e^{A}_{(\alpha,\beta)}\Phi^{-1}= e^{A'}_{(\alpha,\beta)}$ again, where $A'=\left(\begin{array}{cc} c & d \\ a & b \end{array}\right)=M_{\sigma_{2}}A$ in this case.

Hence conjugation by elements of the form $(\mathrm{Id},\tau_{2})$ acts by left multiplication of the corresponding matrix on $A$.

Now taking $\Phi=(\sigma_{1},\mathrm{Id})$ and conjugating an element of $G_{A}$ we have
     \[\Phi e^{A}_{(\alpha,\beta)}\Phi^{-1}
    \left(\left[\begin{tabular}{c}$x_{1}$\\$y_{1}$\\$z_{1}$\\
                 \end{tabular}\right]
     ,\left[\begin{tabular}{c}$x_{2}$\\$y_{2}$\\$z_{2}$\\
            \end{tabular}\right]\right) =
    \left(\left[\begin{tabular}{c}$\xi^{\beta-\alpha}x_{1}$\\$\xi^{-\alpha}y_{1}$\\$z_{1}$\\
                    \end{tabular}\right]
     ,\left[\begin{tabular}{c} $\xi^{a\alpha+b\beta}x_{2}$\\$\xi^{c\alpha+d\beta}y_{2}$\\$z_{2}$\\
            \end{tabular}\right]\right).\]
            Thus $\Phi e^{A}_{(\alpha,\beta)}\Phi^{-1}=e^{A'}_{(\beta-\alpha,-\alpha)}$ where $A'=\left(\begin{array}{cc} b & -a-b \\ d & -c-d \end{array}\right)=AM_{\sigma_{1}}^{-1}$.

Analogously one gets that for $\Phi=(\sigma_{2},\mathrm{Id})$
     \[\Phi e^{A}_{(\alpha,\beta)}\Phi^{-1}
    \left(\left[\begin{tabular}{c}$x_{1}$\\$y_{1}$\\$z_{1}$\\
                 \end{tabular}\right]
     ,\left[\begin{tabular}{c}$x_{2}$\\$y_{2}$\\$z_{2}$\\
            \end{tabular}\right]\right) =
    \left(\left[\begin{tabular}{c}$\xi^{\beta}x_{1}$\\$\xi^{\alpha}y_{1}$\\$z_{1}$\\
                    \end{tabular}\right]
     ,\left[\begin{tabular}{c} $\xi^{a\alpha+b\beta}x_{2}$\\$\xi^{c\alpha+d\beta}y_{2}$\\$z_{2}$\\
            \end{tabular}\right]\right),\]
            and therefore $\Phi e^{A}_{(\alpha,\beta)}\Phi^{-1}=e^{A'}_{(\beta,\alpha)}$ where $A'=\left(\begin{array}{cc} b & a \\ d & c \end{array}\right)=AM_{\sigma_{2}}^{-1}$.

Finally, for the automorphism $\Phi=J$ we have
     \[\Phi e^{A}_{(\alpha,\beta)}\Phi^{-1}
    \left(\left[\begin{tabular}{c}$x_{1}$\\$y_{1}$\\$z_{1}$\\
                 \end{tabular}\right]
     ,\left[\begin{tabular}{c}$x_{2}$\\$y_{2}$\\$z_{2}$\\
            \end{tabular}\right]\right) =
    \left(\left[\begin{tabular}{c} $\xi^{a\alpha+b\beta}x_{1}$\\$\xi^{c\alpha+d\beta}y_{1}$\\$z_{1}$\\
                    \end{tabular}\right]
     ,\left[\begin{tabular}{c}$\xi^{\alpha}x_{2}$\\$\xi^{\beta}y_{2}$\\$z_{2}$\\
                    \end{tabular}\right]\right),\]
                    which means that $\Phi e^{A}_{(\alpha,\beta)}\Phi^{-1}=e^{A^{-1}}_{(a\alpha+b\beta,c\alpha+d\beta)}$.
\end{proof}

\begin{corollary}\label{cor:isomorphic}
Two Beauville surfaces $S_{A}^{n}$ and $S_{B}^{n}$ are isomorphic if and only if $B=M_{\tau_{2}}A^{\pm 1} M_{\tau_{1}}$ for some $(\tau_{1},\tau_{2})\in\mathcal{S}_{3}\times\mathcal{S}_{3}$.
\end{corollary}


As noted above, $|\mathfrak{F}_{5}|=24$. One can explicitly write down the 24 matrices and check by hand that these matrices lie in a single orbit under the action of $\mathrm{Aut}(F_{n}\times F_{n})$. In this way one obtains the following result:

\begin{corollary}Up to isomorphism there is only one Beauville surface with group $\mathbb{Z}_{5}^{2}$, namely $S^{5}_{A}$ where
    \[A=\left(\begin{array}{cc} 1 & 3 \\ 2 & 4 \end{array}\right)\]
\end{corollary}

This result will also follow from the general formula for the number of isomorphism classes of Beauville surfaces with abelian group given in Theorem~\ref{thm:numbersurfaces}.

\

\noindent{\bf Remark 2.} In~\cite{BaCa} Bauer and Catanese state that there are two (non isomorphic) Beauville structures on the product $F_{5}\times F_{5}$ originally considered by Beauville. It seems that this discrepancy is due to the fact that they regard two Beauville surfaces as equivalent if there is a factor-preserving isomorphism between them, whereas here we also consider factor-interchanging isomorphisms such as $J$.

\

In the next section we will discuss the number of isomorphism classes of Beauville surfaces $S_{A}^{n}$ for each $n$.

\section{Isomorphism classes of Beauville surfaces with abelian group}

We are interested in finding the number of isomorphism classes of Beauville surfaces with group $\mathbb{Z}_{n}^{2}$ for a given $n$. By Corollary~\ref{cor:isomorphic} this is equivalent to counting the number of orbits of the group $W=(\mathcal{S}_{3}\times\mathcal{S}_{3})\rtimes\langle J\rangle = \mathcal{S}_3\wr\mathcal{S}_2$ on the set $\mathfrak{F}_{n}$, so by the Cauchy-Frobenius Lemma we deduce that the number of isomorphism classes of Beauville surfaces is
    \[\frac{1}{|W|}\sum_{i} |\mathcal{C}_{i}||\mathrm{Fix}(x_{i})|,\]
    where $i$ indexes the conjugacy classes $\mathcal{C}_{i}$ of $W$ and $x_{i}$ is any element of $\mathcal{C}_{i}$. The conjugacy classes of $W$ are shown in Table~\ref{ta:conjugacy}, where $\sigma_2$ and $\sigma_3$ are elements of order $2$ and $3$ in $S_3$.

\begin{table}[!htb]\begin{center}\begin{tabular}{| c | c | c | c |}
\hline
Conjugacy & Representative & Order & Number of \\
class & & & elements \\ \hline\hline
1 & $(\mathrm{Id},\mathrm{Id})$ & 1 & 1\\
2 & $(\mathrm{Id},\sigma_{2})$ & 2 & 6\\
3 & $(\sigma_{2},\sigma_{2})$ & 2 & 9\\
4 & $(\mathrm{Id},\sigma_{3})$ & 3 & 4\\
5 & $(\sigma_{3},\sigma_{3})$ & 3 & 4\\
6 & $(\sigma_{2},\sigma_{3})$ & 6 & 12\\
7 & $(\mathrm{Id},\mathrm{Id})\cdot J$ & 2 & 6\\
8 & $(\mathrm{Id},\sigma_{2})\cdot J$ & 4 & 18\\
9 & $(\sigma_{2},\sigma_{3}\sigma_{2})\cdot J$ & 6 & 12\\ \hline
\end{tabular} \caption{Conjugacy classes of $W$.} \label{ta:conjugacy}\end{center}\end{table}

First we need to know the cardinality of the set $\mathfrak{F}_{n}$.

\begin{lemma}\label{lem:F_n}Let $n$ be a natural number. Then
    \[|\mathfrak{F}_{n}|=n^{4}\,\prod_{p|n}\left(1-\frac{1}{p}\right)\left(1-\frac{2}{p}\right)\left(1-\frac{3}{p}\right)\left(1-\frac{4}{p}\right)\]
    where $p$ ranges over the distinct primes dividing $n$.
\end{lemma}

\begin{proof}We will first prove the formula for $n=p$ prime. We want to count the number of matrices $A=\bigl( \begin{smallmatrix}a&b\\ c&d\end{smallmatrix} \bigr)$ satisfying condition~\eqref{eq:freeactions}.

Based on the correspondence between matrices $A\in\mathfrak{F}_{n}$ and pairs of triples $((1,0),(0,1),(-1,-1))$, $((a,c),(b,d),(-a-b,-c-d))$, we deduce that counting matrices in $\mathfrak{F}_{n}$ is equivalent to counting second triples.

Now, any triple in $G$ projects to an ordered triple of points in the projective line $P = \mathbb{P}^{1}(\mathbb{Z}_{p})$ formed by the $1-$dimensional subspaces of the vector space $G$, so Beauville structures in $G$ correspond to disjoint pairs of triples in $P$. Conversely, any triple in $P$ is induced by $p-1$ triples in $G$, all scalar multiples of each other. The standard triple in $G$ induces the triple $(0, \infty, 1)\in P$. Having chosen the standard triple as the first triple in $G$, there are $|P|-3 = p-2$ points of $P$ remaining, so there are $(p-2)(p-3)(p-4)$ choices for the second triple in $P$, corresponding to $(p-1)(p-2)(p-3)(p-4)$ triples in $G$.

The case $n=p^{e}$ follows from the following fact. Having fixed the standard triple as the first one, any second triple given by $a,b,c\in\mathbb{Z}_{p^{e}}^{2}$ satisfying~(i)-(iii) will give by reduction modulo $p$ a triple $(a\bmod{p},b\bmod{p},c\bmod{p})$. Conversely, for each triple $(a,b,c)$ with $a,b,c\in\mathbb{Z}_{p}^{2}$, each of the $p^{4e-4}$ choices for $\mathrm{v}=(h_{1},h_{2},j_{1},j_{2})$ with $h_{1},h_{2},j_{1},j_{2}=0,\ldots,p^{e-1}-1$ gives a triple $(a_{\mathrm{v}},b_{\mathrm{v}},c_{\mathrm{v}})$ where
    \begin{align}
        &a_{\mathrm{v}}=a+(h_{1}p,h_{2}p) \nonumber\\
        &c_{\mathrm{v}}=c+(j_{1}p,j_{2}p) \nonumber\\
        &b_{\mathrm{v}}=-a_{\mathrm{v}}-c_{\mathrm{v}}\nonumber
    \end{align}such that $(a_{\mathrm{v}}\bmod{p},b_{\mathrm{v}}\bmod{p},c_{\mathrm{v}}\bmod{p})=(a,b,c)$.

Finally, we can extend the formula to all $n$ by a straightforward application of the Chinese Remainder Theorem.
\end{proof}

We note that our formula for $|\mathfrak{F}_n|$ is equivalent to that given by Garion and Penegini in~\cite[Corollary~3.24]{GaPe} for the function denoted there by $N_n$.

Lemma~3 immediately gives us information about the asymptotic behaviour of $|\mathfrak{F}_n|$ for large $n$. If $n=p^e$ for some prime $p\ge 5$ then
\[\frac{|\mathfrak{F}_n|}{n^4} = \left(1-\frac{1}{p}\right)\left(1-\frac{2}{p}\right)\left(1-\frac{3}{p}\right)\left(1-\frac{4}{p}\right) \geq \frac{4!}{5^4},\]
and $|\mathfrak{F}_n|/n^4 \to 1$ as $p\to \infty$. However, if $n$ is divisible by the first $k$ primes $p_i\ge 5$ then since
\[\lim_{k\to\infty}\prod_i\left(1-\frac{1}{p_i}\right) = 0\]
(see Exercise~9.3 of~\cite{JoJo}) we have $|\mathfrak{F}_n|/n^4 \to 0$ as $k\to \infty$.

Since the isomorphism classes of Beauville surfaces with Beauville group $\mathbb{Z}_{n}^{2}$ correspond to the orbits of $W$ on $\mathfrak{F}_n$, and $|W|=72$, there are at least $|\mathfrak{F}_n|/72$ such classes. We will now apply the Cauchy-Frobenius Lemma to find the exact number of isomorphism classes. For this, we need to calculate how many matrices in $\mathfrak{F}_n$ are fixed by each element of the group.

First of all we note that an element of $W$ fixes a matrix in $\mathfrak{F}_{n}$ if and only if it fixes its components modulo the prime powers $p_{i}^{e_{i}}$ in the factorisation of $n$, so we can again restrict our attention to prime powers.

It is obvious that $x_{1}=(\mathrm{Id},\mathrm{Id})$ fixes every matrix in $\mathfrak{F}_{n}$, so the number of its fixed points is
\[\Theta_{1}(n):=|\mathrm{Fix}(x_{1})|=|\mathfrak{F}_{n}|.\]
On the other hand it is easy to see that the conditions on a matrix $A\in\mathrm{GL}_{2}(\mathbb{Z}_{n})$ to be fixed by an element of the conjugacy classes $\mathcal{C}_2$, $\mathcal{C}_3$, $\mathcal{C}_4$, $\mathcal{C}_6$ or $\mathcal{C}_8$ are incompatible with those defining $\mathfrak{F}_{n}$.

The action of $x_{5}=(\sigma_{3},\sigma_{3})\in\mathcal{C}_5$ sends a matrix $A=\bigl( \begin{smallmatrix}a&b\\ c&d\end{smallmatrix} \bigr)$ to $(\sigma_{3},\sigma_{3})A=\bigl( \begin{smallmatrix}d-b&(a+b)-(c+d)\\ -b&a+b\end{smallmatrix} \bigr)$. In the case of $n=p^{e}$ a prime power, the conditions $c=-b$ and $d=a+b$ leave $p^{2e}(1-1/p)(1-2/p)$ possibilities for the matrix $A=\bigl( \begin{smallmatrix}a&b\\ -b&a+b\end{smallmatrix} \bigr)$, from which we have to remove those with a non-unit determinant. If $p\equiv -1\bmod{3}$ then the equation $a^2+ab+b^2\equiv 0\bmod{p}$ has no solutions, since it is equivalent to $\lambda^{2}+\lambda+1\equiv 0$ and $\mathbb{F}_{p}$ contains no non-trivial third root of unity. Otherwise, for a fixed $a$ there are two non-valid choices for $b$. Therefore
    \[\Theta_{2}(p^{e}):=|\mathrm{Fix}(x_{5})|=\left\{
                                        \begin{array}{ll}
                                          p^{2e}\left(1-\frac{1}{p}\right)\left(1-\frac{2}{p}\right), & \hbox{if $p\equiv -1\bmod{3}$;} \\
                                          p^{2e}\left(1-\frac{1}{p}\right)\left(1-\frac{4}{p}\right), & \hbox{if $p\equiv 1\bmod{3}$.}
                                        \end{array}
                                      \right.
    \]

For $x_{7}=(\mathrm{Id},\mathrm{Id})\cdot J\in\mathcal{C}_7$, the action on $A$ yields $(\mathrm{Id},\mathrm{Id})J(A)=A^{-1}$. The equality $A=A^{-1}$ in $\mathfrak{F}_{p^{e}}$ implies that $a=-d$ and $b(\det(A)+1)=0$ and hence $\det(A)=-1$. The general form for a matrix fixed by this element is therefore $A=\bigl( \begin{smallmatrix}a&b\\ c&-a\end{smallmatrix} \bigr)$, with $\det(A)=-a^2-bc=-1$. There are $p^{e}(1-3/p)$ possibilities for $a\not\equiv 0,\pm1\bmod{p}$, and once $a$ is chosen there are $p^{e}(1-5/p)$ choices left for $b\not\equiv 0,-a,\frac{1-a^{2}}{a},1-a,-1-a\bmod{p}$. The final formula gives
    \[\Theta_{3}(p^{e}):=|\mathrm{Fix}(x_{7})|=p^{2e}(1-3/p)(1-5/p)\]

Finally $x_{9}=(\sigma_{2},\sigma_{3}\sigma_{2})\cdot J\in\mathcal{C}_9$ acts by sending $A$ to $(\sigma_{2},\sigma_{3}\sigma_{2})J(A)=\bigl( \begin{smallmatrix}-c&c-a\\ d&b-d\end{smallmatrix} \bigr)/\det(A)$. The conditions imply that $\det(A)$ is different from $-1$ and $\det(A)=\det(A)^{3}$, hence $\det(A)=1$. Writing down the rest of the conditions one gets $a=-c=-d$ and $b=-2a$, so we are looking for elements $a\in\mathbb{Z}_{p^{e}}$ such that $-3a^{2}\equiv 1\bmod{p^{e}}$. Now $-3$ is a square in $\mathbb{Z}_{p^{e}}$ if and only if it is a quadratic residue modulo $p$ (see e.g.~Thm.~7.14 in~\cite{JoJo}), and this is the case only when $p\equiv 1\bmod{3}$. To see this note that if $p\equiv\varepsilon\bmod{3}$, where $\varepsilon=\pm 1$, by the Law of Quadratic Reciprocity
    \[\left(\frac{3}{p}\right)=\left\{
                      \begin{array}{ll}
                        \varepsilon, & \hbox{if $p\equiv 1\bmod{4}$;} \\
                        -\varepsilon, & \hbox{if $p\equiv 3\bmod{4}$.}
                      \end{array}
                    \right.\]
Therefore $(\frac{-3}{p})=(\frac{-1}{p})(\frac{3}{p})=\varepsilon$. As a consequence
    \[\Theta_{4}(p^{e}):=|\mathrm{Fix}(x_{9})|=\left\{
                                        \begin{array}{ll}
                                          0, & \hbox{if $p\equiv -1\bmod{3}$;} \\
                                          2, & \hbox{if $p\equiv 1\bmod{3}$.}
                                        \end{array}
                                      \right.
    \]

We have thus proved the following:

\begin{theorem}\label{thm:numbersurfaces}Let $n=p_{1}^{e_{1}}\cdot\ldots\cdot p_{s}^{e_{s}}$ be a natural number coprime to 6, where $p_1,\ldots, p_k$ are distinct primes. Then the number of isomorphism classes of Beauville surfaces with Beauville group $\mathbb{Z}_{n}^{2}$ is
    \[\Theta(n)=\frac{1}{72}\left(\Theta_{1}(n)+4\prod_{i=1}^{s}\Theta_{2}(p_{i}^{e_{i}})+ 6\prod_{i=1}^{s}\Theta_{3}(p_{i}^{e_{i}})+ 12\prod_{i=1}^{s}\Theta_{4}(p_{i}^{e_{i}}) \right),\]
    where the functions $\Theta_{i}$ are defined as above.
\end{theorem}

The formulae given above for $\Theta_r(p^e)$ for $r=2, 3$ and $4$ show that the sum on the right-hand side of this equation is dominated by the first term, so that
\[\Theta(n)\sim \frac{1}{72}\Theta_1(n) = \frac{1}{72}|\mathfrak{F}_n|
= \frac{n^4}{72}\,\prod_{p|n}\left(1-\frac{1}{p}\right)\left(1-\frac{2}{p}\right)\left(1-\frac{3}{p}\right)\left(1-\frac{4}{p}\right)\]
as $n\to\infty$. In particular, we have the following special case:

\begin{corollary}\label{cor:numbersurfacesprimepower}
For each prime $p\ge 5$ the number $\Theta(p^e)$ of isomorphism classes of Beauville surfaces with Beauville group ${\mathbb Z}^2_{p^e}$ is given by
\begin{multline*}
\frac{1}{72}\left(p^{4e}-10p^{4e-1}+35p^{4e-2}-50p^{4e-3}+24p^{4e-4}+10p^{2e}-60p^{2e-1}\right.\\
\left.+\,98p^{2e-2}\right)
\end{multline*}
if $p\equiv -1$ mod $(3)$, and by
\begin{multline*}
\frac{1}{72}\left(p^{4e}-10p^{4e-1}+35p^{4e-2}-50p^{4e-3}+24p^{4e-4}+10p^{2e}-68p^{2e-1}\right.\\
\left.+106p^{2e-2}+24\right)\end{multline*}
if $p\equiv 1$ mod $(3)$.
\end{corollary}

When $e=1$ this specialises to
    \[\Theta(p)=\left\{
    \begin{array}{ll}
          \dfrac{1}{72}(p^4 - 10 p^3 + 45 p^2 - 110 p + 122), & \hbox{if $p\equiv -1\bmod{3}$,} \\
          &\\
          \dfrac{1}{72}(p^4 - 10 p^3 + 45 p^2 - 118 p + 154), & \hbox{if $p\equiv 1\bmod{3}$.}
    \end{array}\right.\]


\noindent{\bf Remark 3.} In~\cite{BaCaGr05}, Bauer, Catanese and Grunewald considered the asymptotic behaviour of the number of Beauville surfaces with Beauville group $\mathbb{Z}_{n}^{2}$, where $n$ is coprime to 6. Later, Garion and Penegini~\cite{GaPe} considered a wide range of related counting problems; in particular, they obtained bounds similar to those in~\cite{BaCaGr05} for Beauville groups $G=\mathbb{Z}_n^2$, specifically that $\Theta(n)$ lies between $|\mathfrak{F}_n|/72$ and $|\mathfrak{F}_n|/6$ (Corollary~3.24). Their results are consistent with ours.

\section{The automorphism group of $S_{A}^{n}$}

The calculations in the previous section provide some insight into the automorphism group of a Beauville surface with an abelian Beauville group.

The fact that the automorphisms of $S_{A}^{n}$ lift to automorphisms of $F_{n}\times F_{n}$ shows that all automorphisms of $S_{A}^{n}$ are induced by elements of $N(G_{A})$, the normaliser of $G_{A}$ in $\mathrm{Aut}(F_{n}\times F_{n})$ and, in fact, that $\mathrm{Aut}(S_{A}^{n})\cong N(G_{A})/G_{A}$.

Since clearly
    \[G_{A} \le  \mathbb{Z}_{n}^2 \times \mathbb{Z}_{n}^2 \le N(G_{A})\]
one sees immediately that each surface $S_{A}^{n}$ admits the group $(\mathbb{Z}_{n}^2 \times \mathbb{Z}_{n}^2)/G_{A}$ as a group of automorphisms. Moreover, one can identify the Beauville group $\mathbb{Z}_{n}^2$ with $(\mathbb{Z}_{n}^2 \times \mathbb{Z}_{n}^2)/G_{A}$ by sending a pair $(\alpha,\beta)$ to $((0,0),(\alpha,\beta))\bmod{G_{A}}$, which acts as an automorphism of $S_{A}^{n}$ by multiplying the coordinates of the second factor by the corresponding roots of unity. Actually this simply reflects the known fact (see~\cite{Jon}) that a Beauville surface with group $G$ always has the centre $Z(G)$ of $G$ as a group of automorphisms.

This isomorphism $\mathbb{Z}_{n}^{2}\cong(\mathbb{Z}_{n}^{2}\times\mathbb{Z}_{n}^{2})/G_{A}$ allows a further identification
    \[\mathrm{Aut}(S_{A}^{n})/\mathbb{Z}_{n}^2 = N(G_{A})/(\mathbb{Z}_{n}^2 \times \mathbb{Z}_{n}^2),\]
where the group on the right is simply the stabiliser in $\mathrm{Aut}(F_{n}\times F_{n})/ (\mathbb{Z}_{n}^2 \times \mathbb{Z}_{n}^2)=W$ of the matrix $A$.

In fact we have the following:

%

\begin{proposition}\label{prop:automorphismgroup}Let $H$ be the subgroup of $W$ identified as above with the quotient $\mathrm{Aut}(S_{A}^{n})/\mathbb{Z}_{n}^{2}$.
Then $H\cong \{\mathrm{Id}\}$, $\mathbb{Z}_{2}$, $\mathbb{Z}_{3}$, $\mathbb{Z}_{6}$ or $\mathcal{S}_{3}$.
\end{proposition}

\begin{proof}By the comment above, an element of $W$ will induce an automorphism of $S_{A}^{n}$, and therefore belong to $H$, if and only if it fixes the matrix $A$. By the proof of Theorem~\ref{thm:numbersurfaces}, the only elements with fixed points in $\mathfrak{F}_n$ are those in the conjugacy classes $\mathcal{C}_1$, $\mathcal{C}_5$, $\mathcal{C}_7$ and $\mathcal{C}_9$, so $H\subseteq \mathcal{C}_1\cup\mathcal{C}_5\cup\mathcal{C}_7\cup\mathcal{C}_9$.

Let $H^0=H\cap(\mathcal{S}_{3}\times\mathcal{S}_{3})$, so that $|H:H^0|\le 2$ and $H^0\subseteq \mathcal{C}_1\cup\mathcal{C}_5$. If $\gamma_1, \gamma_2\in\mathcal{C}_5$ and $\gamma_2\ne\gamma_1^{\pm 1}$ then $\gamma_1\gamma_2\in\mathcal{C}_4$, so $\gamma_1\gamma_2\not\in H$; it follows that $|H^0|=1$ or $3$, and hence $|H|=1, 2, 3$ or $6$. The only groups of these orders are those listed in the Proposition.
\end{proof}

Moreover, by the discussion prior to Theorem~\ref{thm:numbersurfaces} and the formula for $\Theta_{4}(p^{e})$, if $n$ is divisible by some prime $p\equiv -1\bmod{3}$ then no element of order 6 in $W$ fixes points of $\mathfrak{F}_{n}$. Similarly, the formula for $\Theta_{3}(p^{e})$, together with the fact that no element of $\mathcal{C}_2\cup\mathcal{C}_3$ has fixed points, ensures that no element of order 2 can belong to $H$ if 5 divides $n$. We therefore have the following:

\begin{corollary}Let $S_{A}^{n}$ be a Beauville surface. If a prime $p\equiv -1\bmod{3}$ divides $n$, then $H\cong \{\mathrm{Id}\}$, $\mathbb{Z}_{2}$, $\mathbb{Z}_{3}$ or $\mathcal{S}_{3}$. If 5 divides $n$ then $H\cong \{\mathrm{Id}\}$ or $\mathbb{Z}_{3}$.
\end{corollary}


For instance, in Beauville's original example~\cite{Bea}, $\mathrm{Aut}(S_{A}^{5})$ is a semidirect product of ${\mathbb Z}_5^2$ by $H\cong{\mathbb Z}_3$.

This analysis of automorphism groups has been extended to more general Beauville surfaces in~\cite{Jon}.

\begin{acknowledgement}The first and third author authors were partially supported by MEC grant MTM2009-11848. The third author was partially supported by an FPU grant of the MICINN.
\end{acknowledgement}


\

\

\footnotesize
\noindent {\sc G. Gonz\'{a}lez-Diez:} \\ Departamento de Matem\'aticas, Universidad Aut\'onoma de Madrid, 28049, Madrid, Spain. \\ email: {\tt gabino.gonzalez@uam.es}

\

\noindent {\sc G. A. Jones:} \\ School of Mathematics, University of Southampton, Southampton SO17 1BJ, U.K. \\ email: {\tt G.A.Jones@maths.soton.ac.uk}

\

\noindent {\sc D. Torres-Teigell:} \\Departamento de Matem\'aticas, Universidad Aut\'onoma de Madrid, 28049, Madrid, Spain. \\ email: {\tt david.torres@uam.es}


\begin{thebibliography}{1}
    \bibitem{BaCa} Bauer, I., Catanese, F.: Some new surfaces with $p_{g} = q = 0$, Proceedings of the Fano Conference Torino, 123--142 (2004).
    \bibitem{BaCaGr05} Bauer, I., Catanese, F., Grunewald, F.: Beauville surfaces without real structures I, Geometric methods in algebra and number theory, 1--42 (2005).
    \bibitem{BaCa08} Bauer, I., Catanese, F.: A volume maximizing canonical surface in 3-space, Comment. Math. Helv. {\bf 83}, no. 2, 387--406 (2008).
    \bibitem{BaCaGr07} Bauer, I., Catanese, F., Grunewald, F.: The classification of surfaces with $p_g = q = 0$ isogenous to a product of curves, Pure Appl. Math. Q. {\bf 4}, no. 2, part 1, 547--586 (2008).
    \bibitem{Bea} Beauville, A.: Surfaces alg\'{e}briques complexes. Ast\'{e}risque, No. 54. Soci\'{e}t\'{e} Math\'{e}matique de France, Paris (1978).
    \bibitem{Cat} Catanese, F.: Fibred surfaces, varieties isogenous to a product and related moduli spaces, Amer. J. Math. {\bf 122}, 1--44 (2000).
    \bibitem{FMP} Fairbairn, B., Magaard, K., Parker, C.: Generation of finite simple groups with an application to groups acting on Beauville surfaces. arXiv:1010.3500 (2010).
    \bibitem{FuGo} Fuertes, Y., Gonz\'{a}lez-Diez, G.: On Beauville structures on the groups $S_{n}$ and $A_{n}$, Math. Z. {\bf 264}, no. 4 , 959--968 (2010).
    \bibitem{FuGoJa} Fuertes, Y., Gonz\'{a}lez-Diez, G., Jaikin, A.: On Beauville surfaces, Groups Geom. Dyn. {\bf 5}, no.~1, 107--119 (2011).
    \bibitem{FuJo} Fuertes, Y., Jones, G. A.: Beauville structures and finite groups, J.~Algebra {\bf 340}, 13--27 (2011). 
    \bibitem{GaLaLu} Garion, S., Larsen, M., Lubotzky, A.: Beauville surfaces and finite simple groups. arXiv:1005.2316 (2010).
   \bibitem{GaPe} Garion, S., Penegini, M.:  Beauville surfaces, moduli spaces and finite groups. arXiv:1107.5534v1 (2011).
    \bibitem{GuMa} Guralnick, R., Malle, G.: Simple groups admit Beauville structures. arXiv:1009.6183 (2010).
    \bibitem{Jon} Jones, G. A.: Automorphism groups of Beauville surfaces. arXiv:1102.3055 (2011).
    \bibitem{JoJo} Jones, G. A., Jones, J. M.: Elementary Number Theory. Springer Undergraduate Mathematics Series. Springer-Verlag London, Ltd., London (1998).
\end{thebibliography}
\end{document}